\RequirePackage[l2tabu, orthodox]{nag}
\documentclass[a4paper,final,11pt]{article}

\usepackage[oneside,DIV=12,pagesize=pdftex,paper=a4]{typearea}
\usepackage{tikz,bm}

\usepackage[final]{microtype}
\makeatletter
\def\MT@register@subst@font{\MT@exp@one@n\MT@in@clist\font@name\MT@font@list
   \ifMT@inlist@\else\xdef\MT@font@list{\MT@font@list\font@name,}\fi}
\makeatother 

\usepackage[small]{titlesec}

\usepackage[sc]{mathpazo}

\usepackage{mparhack,fixltx2e}

\usepackage{amsmath, amssymb, amsfonts, amsthm, mathtools}
\usepackage[all,warning]{onlyamsmath}
\usepackage{fixmath}
\mathtoolsset{centercolon}

\usepackage[latin1]{inputenc}
\usepackage[T1]{fontenc}

\usepackage[draft=false,pdftex]{hyperref}
\hypersetup{
    bookmarks=true,         
    unicode=false,          
    pdftoolbar=true,        
    pdfmenubar=true,        
    pdffitwindow=false,     
    pdfstartview={FitH},    
    pdftitle={Favourite distances in high dimensions},    
    pdfauthor={Konrad J. Swanepoel},     
    pdfnewwindow=false,      
    colorlinks=true,       
    linkcolor=black,          
    citecolor=black,        
    filecolor=black,      
    urlcolor=black           
}

\theoremstyle{plain}
\newtheorem{theorem}{Theorem}
\newtheorem{lemma}[theorem]{Lemma}
\newtheorem{mytheorem}{Theorem}

\newcommand{\divides}{\mid}
\newcommand{\notdivides}{\nmid}
\newcommand{\setbuilder}[2]{\left\{#1\;\colon\,#2\right\}}
\newcommand{\set}[1]{\left\{#1\right\}}
\newcommand{\epsi}{\varepsilon}
\newcommand{\length}[1]{\lvert#1\rvert}
\newcommand{\card}[1]{\left\lvert#1\right\rvert}
\newcommand{\numbersystem}[1]{\mathbb{#1}}
\newcommand{\bN}{\numbersystem{N}}
\newcommand{\bR}{\numbersystem{R}}
\newcommand{\dE}{\vec{E}}
\newcommand{\dG}{\vec{G}}

\DeclareMathOperator{\diam}{diam}
\newcommand{\vect}[1]{\bm{#1}}
\newcommand{\vo}{\vect{o}}
\newcommand{\vs}{\vect{s}}
\newcommand{\vv}{\vect{v}}
\newcommand{\vx}{\vect{x}}
\newcommand{\vy}{\vect{y}}

\newcommand{\define}[1]{\emph{#1}}

\title{Favourite distances in high dimensions}
\author{Konrad J.\ Swanepoel\\[2mm]
{\small Department of Mathematics}\\
{\small London School of Economics and Political Science}\\
{\small Houghton Street, London WC2A 2AE, United Kingdom}\\
{\small Email: \href{mailto:k.swanepoel@lse.ac.uk}{k.swanepoel@lse.ac.uk}}
}
\date{\small 24\textsuperscript{th} August 2011}

\begin{document}
\maketitle

\begin{abstract}
Let $S$ be a set of $n$ points in $\bR^d$.
Assign to each $\vx\in S$ an arbitrary distance $r(\vx)>0$.
Let $e_r(\vx,S)$ denote the number of points in $S$ at distance $r(\vx)$ from $\vx$.
Avis, Erd\H{o}s and Pach (1988) introduced the extremal quantity $f_d(n)=\max\sum_{\vx\in S}e_r(\vx,S)$, where the maximum is taken over all $n$-point subsets $S$ of $\bR^d$ and all assignments $r\colon S\to(0,\infty)$ of distances.

We give a quick derivation of the asymptotics of the error term of $f_d(n)$ using only the analogous asymptotics of the maximum number of unit distance pairs in a set of $n$ points:
\[ f_d(n) = \left(1-\frac{1}{\lfloor d/2\rfloor}\right)n^2 + 
    \begin{cases} \Theta(n) & \text{if $d$ is even,}\\ \displaystyle \Theta((n/d))^{4/3}) & \text{if $d$ is odd.} \end{cases}
\]
The implied constants are absolute.
This improves on previous results of Avis, Erd\H{o}s and Pach (1988) and Erd\H{o}s and Pach (1990).

Then we prove a stability result for $d\geq 4$, asserting that if $(S,r)$ with $\card{S}=n$ satisfies $e_r(S)=f_d(n)-o(n^2)$, then, up to $o(n)$ points, $S$ is a Lenz construction with $r$ constant.
Finally we use stability to show that for $n$ sufficiently large (depending on $d$) the pairs $(S,r)$ that attain $f_d(n)$ are up to scaling exactly the Lenz constructions that maximise the number of unit distance pairs with $r\equiv 1$, with some exceptions in dimension $4$.

Analogous results hold for the furthest neighbour digraph, where $r$ is fixed to be $r(\vx)=\max_{\vy\in S} \length{\vx\vy}$ for $\vx\in S$.
\end{abstract}

\section{Introduction}
Denote the $d$-dimensional Euclidean space by $\bR^d$, and the Euclidean distance between points $\vx$ and $\vy$ by $\length{\vx\vy}$.
Let $S$ be a set of $n$ points in $\bR^d$.
Let $r\colon S\to(0,\infty)$ be a choice of a positive number for each point in $S$.
Define the \define{favourite distance digraph on $S$ determined by $r$} to be the directed graph $\dG_r(S)=(S,\dE_r(S))$ on the set $S$ where
\[\dE_r(S):=\setbuilder{(\vx,\vy)}{\vx,\vy\in S\text{ and }\length{\vx\vy}=r(\vx)}\text{.}\]
Write $e_r(S):=\card{\dE_r(S)}$.
Define
\[ f_d(n) := \max\setbuilder{e_r(S)}{S\subset\bR^d, \card{S}=n \text{ and } r\colon S\to(0,\infty)}\text{.} \]
Define $D=D_S\colon S\to(0,\infty)$ by
\[ D(\vx):=\max\setbuilder{\length{\vx\vs}}{\vs\in S}\text{.}\]
Then $\dG_D(S)$ is called the \define{furthest neighbour digraph} of $S$.
Define
\[ g_d(n) :=\max\setbuilder{e_D(S)}{S\subset\bR^d, \card{S}=n}\text{.}\]
Clearly $g_d(n)\leq f_d(n)$.
In fact $g_d(n)\sim f_d(n)\sim(1-1/\lfloor d/2\rfloor)n^2$ for any fixed $d\geq 4$ as $n\to\infty$ \cite{AEP, EP}.
A set $S$ of $n$ points and a function $r\colon S\to(0,\infty)$ define an \define{extremal favourite distance digraph} if $e_r(S)=f_d(\card{S})$.
Likewise, $S$ defines an \define{extremal furthest neighbour digraph} if $e_D(S)=g_d(\card{S})$.

\subsection*{Overview}
In this paper we prove a structure theorem for extremal favourite distance digraphs and furthest neighbour digraphs (Theorem~\ref{theorem:higherd}) for dimension $d\geq 4$.
This structure theorem follows from a stability result describing the pairs $(S,r)$ for which $e_r(S)$ is close to $f_d(n)$ (Theorem~\ref{theorem:4dstability}).
In Section~\ref{section1.1} we start off with an easy derivation of the optimal asymptotics of the error term of $f_d(n)$ (Theorem~\ref{theorem:asymptotics}).
This simple proof introduces the basic approach used in this paper.
Section~\ref{section1.2} gives a description of the Lenz configurations and formally states Theorem~\ref{theorem:higherd}.
Then we state Theorem~\ref{theorem:4dstability} in Section~\ref{sub:stability}.
Section~\ref{proof1} contains the proof of Theorem~\ref{theorem:4dstability} and Section~\ref{proof2} the proof of Theorem~\ref{theorem:higherd}.

\bigskip
Note that we only consider dimensions $d\geq 4$ in this paper.
For lower dimensions we only make the following remarks.
The current best estimates \[\frac{n^2}{4}+\frac{5n}{2}-6\leq f_3(n)\leq\frac{n^2}{4}+\frac{5n}{2}+6\]
for large $n$ can be found in another paper \cite{Sw-extremal3d}.
Csizmadia \cite{Cs} determined $g_3(n)$ exactly for large $n$.
In dimension $2$ a construction gives $f_2(n)=\Omega(n^{4/3})$ \cite[p.~187]{BMP}, while the best known upper bound $f_2(n)=O(n^{15/11+\epsi})$ is due to Aronov and Sharir \cite{AS}.
Avis \cite{Avis} and Edelsbrunner and Skiena \cite{ES} determined $g_2(n)$ exactly.

\bigskip
Throughout this paper $[k]$ denotes the set $\set{1,2,\dots,k}$, $\binom{S}{2}$ the set of unordered pairs of elements of $S$, $K_p$ the complete graph on $p$ vertices, and $K_p(t)$ the complete $p$-partite graph with $t$ elements in each class.

\section{Asymptotics}\label{section1.1}
The problem of determining $f_d(n)$ and $g_d(n)$ was originally introduced by Avis, Erd\H{o}s and Pach \cite{AEP}.
They determined $f_d(n)$ asymptotically for even $d\geq 4$.
Erd\H{o}s and Pach \cite{EP} finished off the case of odd $d\geq 5$.
\begin{theorem}[Avis-Erd\H{o}s-Pach \cite{AEP}, Erd\H{o}s-Pach \cite{EP}]\label{theorem:EPasymptotics}
For any $d\geq 4$,
\[f_d(n)=\left(1-\frac{1}{\lfloor d/2\rfloor}+o(1)\right)n^2\text{.}\]
\end{theorem}
We note that for even dimensions $d\geq 4$, the error term in \cite{AEP} is $O(n^{2-\epsi})$ for some $\epsi>0$ independent of $d$.
The lower bound is obtained from the corresponding lower bound for the maximum number $u_d(n)$ of unit distance pairs in a set of $n$ points in $\bR^d$ (the \emph{Lenz construction} \cite{Erdos60}; see Section~\ref{section1.2}).
For any set $S\subset\bR^d$ of $n$ points, let
\[ u(S) := \card{\setbuilder{\set{\vx,\vy}}{\vx,\vy\in S \text{ and } \length{\vx\vy}=1}}\]
and set
\[ u_d(n) := \max\setbuilder{u(S)}{S\subset\bR^d \text{ and } \card{S}=n}\text{.} \]
Clearly $f_d(n)\geq 2u_d(n)$.
Similarly, $g_d(n)\geq 2M_d(n)$, where $M_d(n)$ is the maximum number of diameter pairs in a set of $n$ points in $\bR^d$, defined by
setting
\[ M(S) := \card{\setbuilder{\set{\vx,\vy}}{\vx,\vy\in S \text{ and } \length{\vx\vy}=\diam(S)}}\]
and
\[ M_d(n) := \max\setbuilder{M(S)}{S\subset\bR^d, \card{S}=n}\text{.} \]
We show that the determination of $f_d(n)$ and $g_d(n)$ in effect reduces to the unit distance problem when $d\geq 4$.
A first indication of this is a simple derivation of an asymptotic upper bound for $f_d(n)$ (Theorem~\ref{theorem:asymptotics} below) using only the analogous upper bounds for $u_d(n)$ stated in the following theorem.
\begin{theorem}[Erd\H{o}s \cite{Erdos67}, Erd\H{o}s-Pach \cite{EP}]\label{cor:udextremal}
There exist constants $c_1,c_2>0$ such that for each $d\geq 4$ and all $n\in\bN$,
\[ u_d(n)\leq \frac{1}{2}\left(1-\frac{1}{\lfloor d/2\rfloor}\right)n^2 + 
    \begin{cases} c_1n & \text{if $d$ is even,}\\ \displaystyle c_2\left(\frac{n}{d}\right)^{4/3} & \text{if $d$ is odd.} \end{cases}
\]
\end{theorem}
The above bounds are tight up to the values of $c_1$ and $c_2$ \cite{EP}.
In fact Erd\H{o}s proved that for even $d\geq 4$ and sufficiently large $n$,
\[\frac{1}{2}\left(1-\frac{2}{d}\right)n^2 + n-\frac{d}{2} \leq u_d(n) \leq \frac{1}{2}\left(1-\frac{2}{d}\right)n^2 + n\text{.}\]
However, in the proof of the next theorem, we need a bound which holds for all $n\in\bN$.
Since $f_d(n)\geq 2u_d(n)$, the bounds in the next theorem are also tight up to the values of the constants.
\begin{mytheorem}\label{theorem:asymptotics}
With the same constants $c_1,c_2>0$ as in Theorem~\ref{cor:udextremal}, for each $d\geq 4$ and all $n\in\bN$,
\[ f_d(n)\leq \left(1-\frac{1}{\lfloor d/2\rfloor}\right)n^2 + 
    \begin{cases} 2c_1n & \text{if $d$ is even,}\\ \displaystyle 2c_2\left(\frac{n}{d}\right)^{4/3} & \text{if $d$ is odd.} \end{cases}
\]
\end{mytheorem}
\begin{proof}
Let $S\subset\bR^d$ be an arbitrary set of $n$ points and $r\colon S\to(0,\infty)$ any function that assigns a positive real number to each point in $S$.
We next introduce notation and terminology that will also be used in later proofs.
We first decompose $\dG_r(S)$ into two ordinary graphs.
Let $G^1_r(S)=(S,E^1_r)$ be the graph of \define{single edges}, where 
\[E^1_r:=\setbuilder{\set{\vx,\vy}}{(\vx,\vy)\in\dE_r(S),(\vy,\vx)\notin\dE_r(S)}\text{.}\]
Let $G^2_r(S)=(S,E^2_r)$ be the graph of \define{double edges}, where
\[E^2_r:=\setbuilder{\set{\vx,\vy}}{(\vx,\vy),(\vy,\vx)\in\dE_r(S)}\text{.}\]
Write the connected components of $G^2_r(S)$ as $G^2_r[S_i]$ for $i\in[k]$, where $\set{S_1,\dots,S_k}$ partitions $S$.
For subsets $A,B\subseteq S$, let
\[\dE_r(A,B):=\setbuilder{(\vx,\vy)\in\dE_r(S)}{\vx\in A, \vy\in B}\]
and $e_r(A,B):=\card{\dE_r(A,B)}$.
Write $n_i:=\card{S_i}$ for each $i\in[k]$.

The proof is based on the following two simple facts.
\begin{enumerate}
\item Each $G_r^2[S_i]$ is a scaling of a unit distance graph.
Therefore, $e_r(S_i)\leq u_d(n_i)$.
\item There can only be single edges between different $S_i$.
Consequently, \[e_r(S_i,S_j)+e_r(S_j,S_i)\leq n_in_j\quad\text{for any distinct $i,j$.}\]
\end{enumerate}
The proof is finished by a calculation.
Note that
\begin{align*}
e_r(S) & =\sum_{i=1}^k e_r(S_i)+\sum_{\set{i,j}\in\binom{[k]}{2}}(e_r(S_i,S_j)+e_r(S_j,S_i))\\
&\leq \sum_{i=1}^k 2u_d(n_i)+\sum_{\set{i,j}\in\binom{[k]}{2}}n_in_j\text{.}\\
\intertext{Now fix $S$ and $r$ so that $f_d(n)=e_r(S)$, and apply Theorem~\ref{cor:udextremal} to obtain for odd dimensions $d\geq5$ that}
f_d(n) & \leq \sum_{i=1}^k\left(2u_d(n_i)-\frac{1}{2}n_i^2\right)+\sum_{i=1}^k\frac{1}{2}n_i^2+\sum_{\set{i,j}\in\binom{[k]}{2}}n_in_j\\
& = \sum_{i=1}^k\left(2u_d(n_i)-\frac{1}{2}n_i^2\right) + \frac{1}{2}n^2\\
& \leq\sum_{i=1}^k \left(\left(1-\frac{1}{\lfloor d/2\rfloor}\right)n_i^2 + 2c_2\left(\frac{n_i}{d}\right)^{4/3} -\frac{1}{2}n_i^2\right) + \frac{1}{2}n^2\\
& =\sum_{i=1}^k \left(\left(\frac{1}{2}-\frac{1}{\lfloor d/2\rfloor}\right)n_i^2 + 2c_2\left(\frac{n_i}{d}\right)^{4/3}\right) + \frac{1}{2}n^2\\
& \leq\left(\frac{1}{2}-\frac{1}{\lfloor d/2\rfloor}\right)\left(\sum_{i=1}^k n_i\right)^2 + 2c_2\left(\frac{\sum_{i=1}^k n_i}{d}\right)^{4/3} + \frac{1}{2}n^2\\
& =\left(1-\frac{1}{\lfloor d/2\rfloor}\right)n^2 + 2c_2\left(\frac{n}{d}\right)^{4/3} \text{,}
\end{align*}
where the last inequality follows from the inequality
\begin{equation}\label{alphaineq}
\sum_{i=1}^k n_i^\alpha\leq\left(\sum_{i=1}^k n_i\right)^{\alpha}\quad\text{for all $n_i\geq 0$ and $\alpha\geq 1$,}
\end{equation}
which is easily seen to be true (for example from Minkowski's inequality).
The calculation for even values of $d\geq 4$ is similar.
\end{proof}

\section{Extremal configurations}\label{section1.2}
By a \define{Lenz configuration for distance $\lambda>0$} we mean a finite set of the following type \cite{Erdos60, Brass97}.

If $d\geq 4$ is even, let $p=d/2$ and consider any orthogonal decomposition $\bR^d=V_1\oplus\dots\oplus V_p$ with all $V_i$ $2$-dimensional.
In each $V_i$, let $C_i$ be the circle with centre the origin $\vo$ and radius $r_i$, such that $r_i^2+r_j^2=\lambda^2$ for all distinct $i$ and $j$.
When $d\geq 6$ this implies that each $r_i=\lambda/\sqrt{2}$.
We call the $p$ circles $(C_1,\dots,C_p)$ an \define{even-dimensional Lenz system}.
Define an \define{even-dimensional Lenz configuration for the distance $\lambda$} to be any finite subset $S$ of some translate $\vv+\bigcup_{i=1}^p C_i$ of the circles.
The \define{partition associated to the Lenz configuration $S$} is the partition induced by the circles, i.e.\ the $p$ subsets $S_1,\dots,S_p$ where $S_i=S\cap(\vv+C_i)$.

If $d\geq 5$ is odd, let $p=\lfloor d/2\rfloor$, and consider any orthogonal decomposition $\bR^d=V_1\oplus\dots\oplus V_p$ with $V_1$ $3$-dimensional and all other $V_i$ ($i=2,\dots,p$) $2$-dimensional.
Let $\Sigma_1$ be the $2$-sphere in $V_1$ with centre $\vo$ and radius $r_1$, and for each $i=2,\dots,p$, let $C_i$ be the circle with centre $\vo$ and radius $r_i$, such that $r_i^2+r_j^2=\lambda^2$ for all distinct $i, j$.
When $d\geq 7$, necessarily each $r_i=\lambda/\sqrt{2}$.
We call the $2$-sphere and $p-1$ circles $(\Sigma_1,C_2,\dots,C_p)$ an \define{odd-dimensional Lenz system}.
We define an \define{odd-dimensional Lenz configuration for the distance $\lambda$} to be any finite subset of some translate $\vv+\left(\Sigma_1\cup\bigcup_{i=2}^p C_i\right)$ of the $2$-sphere and circles.
The \define{partition associated to the Lenz configuration $S$} is the partition induced by the $2$-sphere and circles, i.e.\ the $p$ subsets $S_1,\dots,S_p$ where $S_1=S\cap\Sigma_1$ and for $i\geq 2$, $S_i=S\cap(\vv+C_i)$.
The following theorem
states that the extremal sets for unit distances and for diameters are Lenz configurations, at least for a sufficiently large number of points.

\begin{theorem}[\cite{Brass97, Sw09}]\label{theorem:udextremal}
For any $d\geq 4$ there exists $n_0\in\bN$ such that any set $S$ for which $\card{S}=n\geq n_0$ and such that $u(S)=u_d(n)$ is a Lenz configuration for the distance $1$.

For any $d\geq 4$ there exists $n_0\in\bN$ such that any set $S$ for which $\card{S}=n\geq n_0$ and such that $M(S)=M_d(n)$ is a Lenz configuration for the distance $\diam(S)$.
\end{theorem}
As a corollary of the main result of this paper (Theorem~\ref{theorem:4dstability} in Section~\ref{sub:stability}) we show that when $d\geq 4$, the extremal favourite distance digraphs (furthest neighbour digraphs) are exactly the same as the sets for which $u_d(n)$ ($M_d(n)$ respectively) is maximised, for all sufficiently large $n$, depending on $d$, except when $d=4$, where there is an exceptional construction for all sufficiently large $n\equiv 1\pmod{8}$.

\begin{mytheorem}\label{theorem:higherd}
For any $d\geq 4$ there exists $n_0\in\bN$ such that the following holds.
\begin{enumerate}
\item Let $S\subset\bR^d$ and a function $r\colon S\to(0,\infty)$ be given for which $\card{S}=n\geq n_0$ and $e_r(S)=f_d(n)$.
Then $r\equiv c$ for some $c>0$ and $S$ is a Lenz configuration for the distance $c$, except when $d=4$ and $n-1$ is divisible by $8$, where the following situation is also possible:
for some $a\in S$ and $c>0$, $S\setminus\set{a}$ is a Lenz configuration for the distance $c$ on two circles $C_1$ and $C_2$ of equal radius $c/\sqrt{2}$, $a$ is the common centre of the two circles, $C_i\cap S$ consists of the vertices of $(n-1)/8$ squares inscribed in $C_i$ \textup{(}$i=1,2$\textup{)}, and $\left.r\right|_{S\setminus\set{a}}\equiv c$, $r(a)=c/\sqrt{2}$.
\item Let $S\subset\bR^d$ be given for which $\card{S}=n\geq n_0$ and $e_D(S)=g_d(n)$.
Then $r\equiv\diam(S)$ and $S$ is a Lenz configuration for the distance $\diam(S)$.
\end{enumerate}
In particular, $f_d(n)=2u_d(n)$ and $g_d(n)=2M_d(n)$ for all $d\geq 4$ and $n\geq n_0(d)$.
\end{mytheorem}
Note that the exact values of $u_d(n)$ for even $d\geq 4$ and of $M_d(n)$ for all $d\geq 4$ are known, at least for sufficiently large $n$ \cite{Brass97, Sw09}; see Lemmas~\ref{exactvalues1} and \ref{exactvalues2} for some of these values.

\section{Stability}\label{sub:stability}
The following theorem states that if the number of unit distance pairs of points from $S\subset\bR^d$, with $n:=\card{S}$ is sufficiently large, is within $o(n^2)$ of the maximum $u_d(n)$, then $S$ is a Lenz configuration up to $o(n)$ points.
\begin{theorem}[\cite{Sw09}]\label{theorem:udstability}
For any $d\geq 4$ and $\epsi>0$ there exist $\delta>0$ and $n_0\in\bN$ such that for any set $S$ with $\card{S}=n\geq n_0$ that satisfies
\[u(S)>\frac{1}{2}\left(1-\frac{1}{p}-\delta\right)n^2\quad\text{\textup{(}where $p=\lfloor d/2\rfloor$\textup{)},}\]
there exists a subset $T\subseteq S$ such that $\card{T}<\epsi n$ and $S\setminus T$ is a Lenz configuration.
Furthermore, the partition $S_1,\dots,S_p$ of $S\setminus T$ associated to the Lenz configuration satisfies
\[ \frac{n}{p}-\epsi n < \card{S_i} < \frac{n}{p} +\epsi n \quad \text{for all $i\in[p]$.}\]
\end{theorem}
The next theorem is an analogue of the above theorem for favourite distance digraphs.
\begin{mytheorem}\label{theorem:4dstability}
For any $d\geq 4$ and any $\epsi>0$ there exist $\delta>0$ and $n_0\in\bN$ such that for any $S\subset\bR^d$ with $\card{S}=n\geq n_0$ and any $r\colon S\to(0,\infty)$ that satisfy
\[e_r(S)>\left(1-\frac{1}{p}-\delta\right)n^2\quad\text{\textup{(}where $p=\lfloor d/2\rfloor$\textup{)},}\]
there exist $T\subseteq S$ and $c>0$ such that $\card{T}<\epsi n$, $S\setminus T$ is a Lenz configuration with distance $c$, and $\left. r\right|_{S\setminus T}\equiv c$.
Furthermore, the partition $S_1,\dots,S_p$ of $S\setminus T$ associated to the Lenz configuration satisfies
\[ \frac{n}{p}-\epsi n < \card{S_i} < \frac{n}{p} +\epsi n \quad \text{for all $i\in[p]$.}\]
\end{mytheorem}
By applying Theorem~\ref{theorem:udstability}, the above theorem is relatively easy to prove for $d\geq 6$, but surprisingly, takes some work in the cases $d\in\set{4,5}$.
This is not so much because the Lenz construction is slightly more complicated in dimensions $4$ and $5$, but rather due to certain complications in the extremal theory of digraphs not shared by the extremal theory of ordinary graphs \cite{BS02}.

\section{Proof of Theorem~\ref{theorem:4dstability}}\label{proof1}
Let $d\geq 4$ and $\epsi>0$ be given.
Without loss of generality $\epsi<\frac{1}{20}$.
Let $p=\lfloor d/2\rfloor$.
We take $\delta>0$ to be sufficiently small depending only on $\epsi$ and $d$.
In particular, we need 
\begin{itemize}
\item $\delta < \epsi^2/144$ and,
\item after an application of stability for unit distances (Theorem~\ref{theorem:udstability}), we may also assume that $\delta$ has been chosen so that for some $N\in\bN$ (which we now fix), for any $S\subset\bR^d$ with $\card{S}=n\geq N$, if $u(S)>\frac{1}{2}\left(1-\frac{1}{p}-32\delta\right)n^2$, then for some $T\subseteq S$ of size $\card{T}<\epsi n/3$, $S\setminus T$ is a Lenz configuration such that the number of elements in each part of the associated partition is in the interval $((1/p-\epsi/3)n,(1/p+\epsi/3)n)$.
\end{itemize}
We also take $n_0\in\bN$ sufficiently large depending only on $\epsi$, $d$ and $\delta$, as follows.
We need
\begin{itemize}
\item $n_0>9/\delta$ and $n_0\geq 4N$,
\item $n_0$ to be sufficiently large such that for all $n\geq n_0$, $f_3(n-2)+4n<\left(\frac{1}{2}-\delta\right)n^2$ (by Avis-Erd\H{o}s-Pach \cite{AEP} $f_3(n)=\frac{n^2}{4}+O(n^{2-c})$; in \cite{Sw-extremal3d} we show $f_3(n)\leq \frac{n^2}{4}+\frac{5n}{2}+6$ for sufficiently large $n$), 
\item $n_0>(2c_2/\delta)^{3/2}d^{-2}$, where $c_2$ is the constant from Theorems~\ref{cor:udextremal} and \ref{theorem:asymptotics},
\item $n_0$ to be sufficiently large such that for all $n\geq \epsi n_0/4$, $f_5(n)<\left(\frac{1}{2}+\delta\right)n^2$ (Theorem~\ref{theorem:EPasymptotics}), and
\item $n_0$ to be sufficiently large such that the Erd\H{o}s-Stone theorem guarantees that any graph on $n\geq n_0$ vertices and at least $\left(\frac{1}{3}+\delta\right)n^2$ edges contains a $K_4(p_0)$, where $p_0$ is a constant such that no orientation of $K_4(p_0)$ can be a subgraph of a favourite distance digraph in $\bR^5$ (Lemmas~4 and 5 in \cite{AEP}).
\end{itemize}
Choose $S\subset\bR^d$ with $\card{S}=n\geq n_0$ and $r\colon S\to(0,\infty)$ such that $e_r(S)>\left(1-\frac{1}{p}-\delta\right)n^2$.
We continue with the notation established in the proof of Theorem~\ref{theorem:asymptotics}.
Thus let $G^1_r(S)$ be the graph of single edges and $G^2_r(S)$ the graph of double edges of $\dG_r(S)$ with connected components $G^2_r[S_i]$ ($i\in[k]$).
As before, $n_i=\card{S_i}$.
Also write \[d_{i,j}:=\frac{e_r(S_i,S_j)}{n_in_j}\] for distinct $i,j\in[k]$
and $\alpha_i:=n_i/n$.
Similar to the calculation in the proof of Theorem~\ref{theorem:asymptotics},
\begin{align*}
e_r(S) & =\sum_{i=1}^k e_r(S_i)+\sum_{\set{i,j}\in\binom{[k]}{2}}(e_r(S_i,S_j)+e_r(S_j,S_i))\\
&\leq \sum_{i=1}^k 2u_d(n_i) +\sum_{\set{i,j}\in\binom{[k]}{2}}(d_{i,j}+d_{j,i})\alpha_i\alpha_jn^2\\
&\leq \sum_{i=1}^k \left(\left(1-\frac{1}{p}\right)(\alpha_i n)^2 + 2c_2\left(\frac{\alpha_i n}{d}\right)^{4/3}\right) +\sum_{\set{i,j}\in\binom{[k]}{2}}(d_{i,j}+d_{j,i})\alpha_i\alpha_jn^2\text{.}
\end{align*}
It is given that $e_r(S)>\left(1-\frac{1}{p}-\delta\right)n^2$.
Therefore,
\begin{align}
1-\frac{1}{p}-\delta &< \sum_{i=1}^k \left(1-\frac{1}{p}\right)\alpha_i^2 + 2c_2d^{-4/3}n^{-2/3}\sum_{i=1}^k\alpha_i^{4/3} + \sum_{\set{i,j}\in\binom{[k]}{2}}(d_{i,j}+d_{j,i})\alpha_i\alpha_j \notag \\
&\leq \left(1-\frac{1}{p}\right)\sum_{i=1}^k\alpha_i^2 + \delta\left(\sum_{i=1}^k\alpha_i\right)^{4/3} + \sum_{\set{i,j}\in\binom{[k]}{2}}(d_{i,j}+d_{j,i})\alpha_i\alpha_j \label{1}\\
& \qquad\qquad\qquad\text{(since $n$ is sufficiently large and using \eqref{alphaineq})}\notag\\
& = 1-\frac{1}{p} + \delta - \sum_{\set{i,j}\in\binom{[k]}{2}}\left(2\left(1-\frac{1}{p}\right)-d_{i,j}-d_{j,i}\right)\alpha_i \alpha_j\text{.} \notag
\end{align}
Therefore,
\begin{equation}\label{2}
\sum_{\set{i,j}\in\binom{[k]}{2}}\left(2\left(1-\frac{1}{p}\right)-d_{i,j}-d_{j,i}\right)\alpha_i \alpha_j < 2\delta\text{.}
\end{equation}
As noted in the proof of Theorem~\ref{theorem:asymptotics}, there are no double edges between $S_i$ and $S_j$ when $i\neq j$.
Consequently, $d_{i,j}+d_{j,i}\leq 1$, and therefore,
\[ \sum_{\set{i,j}\in\binom{[k]}{2}}\left(1-\frac{2}{p}\right)\alpha_i\alpha_j < 2\delta\text{.}\]
Assume for the moment that $d\geq 6$.
Then $p\geq 3$, hence $\sum_{\set{i,j}}\alpha_i\alpha_j<6\delta$.
Substituting back into \eqref{1} we obtain
\[ 1 - \frac{1}{p}-\delta < \left(1-\frac{1}{p}\right)\sum_{i=1}^k\alpha_i^2 + \delta+ 6\delta\text{,}\]
which gives
\[ \sum_{i=1}^k\alpha_i^2 > \frac{1-\frac{1}{p}-8\delta}{1-\frac{1}{p}}\geq 1-12\delta\text{.}\]
Since $\sum_{i=1}^k\alpha_i=1$, it follows that $\alpha_i>1-12\delta$ for some $i\in[k]$. 
Without loss of generality, $\alpha_1>1-12\delta$.
Calculating again,
\begin{align}
\left(1-\frac{1}{p}-\delta\right)n^2 &< e_r(S) = e_r(S_1)+\sum_{i=2}^ke_r(S_i)+\sum_{\set{i,j}\in\binom{[k]}{2}}\alpha_i\alpha_j n^2\label{recalc}\\
&< e_r(S_1)+\sum_{i=2}^n(\alpha_i n)^2+6\delta n^2\notag\\
&\leq e_r(S_1)+\left(\sum_{i=2}^n\alpha_i\right)^2n^2+6\delta n^2\quad\text{(by \eqref{alphaineq})}\notag\\
&< e_r(S_1)+(12\delta)^2n^2+6\delta n^2\text{,}\notag
\end{align}
and assuming after scaling that $\left.r\right|_{S_1}\equiv 1$, we obtain
\begin{equation*}
2u(S) \geq e_r(S_1) > \left(1-\frac{1}{p}-7\delta-(12\delta)^2\right)n^2
> \left(1-\frac{1}{p}-32\delta\right)n^2\text{.}
\end{equation*}
By the choice of $\delta$ and $n_0$, the proof is concluded by an application of Theorem~\ref{theorem:udstability}.
This establishes the theorem for all dimensions $d\geq 6$.

The remaining cases are $d=4$ and $d=5$.
The $4$-dimensional case of the theorem is implied by the $5$-dimensional case.
In fact, the theorem for $d=5$ implies that when $S\subset\bR^5$ is contained in an affine hyperplane $H$, then for some $T\subseteq S$ with $\card{T}<\epsi \card{S}$, $S\setminus T$ is the intersection of a $5$-dimensional Lenz configuration with $H$.
Such an intersection is clearly either a $4$-dimensional Lenz configuration, or becomes $3$-dimensional after removing at most $2$ points.
In the latter case
\begin{equation*}
e_r(S) \leq f_3(n-2) + 2(n-2)+2(n-1) < \left(\frac{1}{2}-\delta\right)n^2
\end{equation*}
by choice of $n_0$.
Thus the former case necessarily occurs.

For the remainder of the proof assume that $d=5$.
Then $p=2$ and \eqref{2} can now be written as
\[
\sum_{\set{i,j}\in\binom{[k]}{2}}\alpha_i \alpha_j < \sum_{\set{i,j}\in\binom{[k]}{2}}(d_{i,j}+d_{j,i})\alpha_i\alpha_j +2\delta\text{.}
\]
Thus the graph of single edges $G_r^1(S)$ is almost the complete $k$-partite graph with classes $S_1,\dots,S_k$.
We next apply the Erd\H{o}s-Stone theorem to show that one of the $S_i$ is large in the sense that $\card{S_i}=\Omega(n)$.
(We have no control over $k$ yet, and have to eliminate the possibility that $k$ is large with each $S_i$ small, which would imply that $\dG_r(S)$ is close to a tournament---a case which would be difficult to handle geometrically).
Since $n_0$ and $p_0$ were chosen so that $G_r^1(S)$ does not contain a copy of $K_4(p_0)$, the Erd\H{o}s-Stone theorem gives for sufficiently large $n$ that
\[
\left(\frac{1}{3}+\delta\right)n^2 > \card{E(G_r^1)} = \sum_{\set{i,j}\in\binom{[k]}{2}}(d_{i,j}+d_{j,i})\alpha_i \alpha_j n^2\text{.}
\]
Therefore, $\sum_{\set{i,j}}\alpha_i \alpha_j < \frac{1}{3}+3\delta$, and
\[\sum_{i=1}^k\alpha_i^2 = \left(\sum_{i=1}^k\alpha_i\right)^2-2\sum_{\set{i,j}\in\binom{[k]}{2}}\alpha_i \alpha_j > 1 - 2\left(\frac{1}{3}+3\delta\right)=\frac{1}{3}-6\delta\text{.}\]
It follows that for some $i\in[k]$, $\alpha_i>\frac{1}{3}-6\delta$.
Without loss of generality, $\alpha_1>\frac{1}{3}-6\delta>\frac{1}{4}$.
Thus $\card{S_1}>n/4$.
This enables us to show next that $S_1$ is almost a Lenz configuration.
Suppose to the contrary that $e_r(S_1)\leq \left(\frac{1}{2}-32\delta\right)(\alpha_1 n)^2$.
Starting off as in \eqref{recalc}, an application of Theorem~\ref{cor:udextremal} now gives the following:
\begin{align*}
\left(\frac{1}{2}-\delta\right)n^2 &< e_r(S) = e_r(S_1)+\sum_{i=2}^ke_r(S_i)+\sum_{\set{i,j}\in\binom{[k]}{2}}\alpha_i\alpha_j n^2\\
&< \left(\frac{1}{2}-32\delta\right)(\alpha_1n)^2 + \sum_{i=2}^k\left(\frac{1}{2}(\alpha_i n)^2 + 2c_2\left(\frac{\alpha_i n}{5}\right)^{4/3}\right)\\
&\qquad +\sum_{\set{i,j}\in\binom{[k]}{2}}\alpha_i\alpha_j n^2\text{.}
\end{align*}
It follows that
\begin{align*}
\frac{1}{2}-\delta &< \frac{1}{2}\sum_{i=1}^k\alpha_i^2 - 32\delta\alpha_1^2 + \sum_{i=2}^k \frac{2c_2}{5^{4/3}n^{2/3}}\alpha_i^{4/3}+\sum_{\set{i,j}\in\binom{[k]}{2}}\alpha_i\alpha_j\\
&< \frac{1}{2}\sum_{i=1}^k\alpha_i^2 - 32\delta\alpha_1^2 + \delta\sum_{i=2}^k\alpha_i^{4/3} +\sum_{\set{i,j}\in\binom{[k]}{2}}\alpha_i\alpha_j\quad\text{for $n$ sufficiently large}\\
&= \frac{1}{2}  - 32\delta\alpha_1^2 + \delta\sum_{i=2}^k\alpha_i^{4/3} < \frac{1}{2}  - 32\delta\alpha_1^2 + \delta\left(\sum_{i=2}^k\alpha_i\right)^{4/3}\quad\text{by \eqref{alphaineq}}\\
&< \frac{1}{2} -32\delta\alpha_1^2+\delta < \frac{1}{2}-\delta
\end{align*}
since $\alpha_1>1/4$.
This contradiction gives (after scaling so that $\left.r\right|_{S_1}\equiv1$) that
\[ 2u(S_1) = e_r(S_1)>\left(\frac{1}{2}-32\delta\right)(\alpha_1 n)^2\text{.}\]
Since $\card{S_1}=\alpha_1 n>n_0/4\geq N$ and by the choice of $\delta$ and $n_0$, Theorem~\ref{theorem:udstability} gives a $T\subset S_1$ with $\card{T}<\epsi\card{S_1}/3\leq\epsi n/3$ such that $S_1\setminus T$ is a Lenz configuration for the distance $1$.
Thus we may write $\bR^5=V_1\oplus V_2$ with $\dim V_1=3$ and $\dim V_2=2$ such that $S_1\setminus T\subset \Sigma_1\cup C_2$, where $\Sigma_1$ is a $2$-sphere in $V_1$ with centre $\vo$ and radius $r_1$, and $C_2$ is a circle in $V_2$ with centre $\vo$ and radius $r_2$, where $r_1^2+r_2^2=1$.
After possibly replacing $T$ by a subset, we may assume without loss of generality that $T\cap(\Sigma_1\cup C_2)=\emptyset$.
Also then $\card{S_1\cap\Sigma_1},\card{S_1\cap C_2}<(\frac{1}{2}+\frac{2\epsi}{3})\card{S_1}$.
Since $S_1\setminus T$ is a Lenz configuration and $\card{T}<\epsi n/3$, the proof would be finished if we can show that $\card{S\setminus S_1}<2\epsi n/3$.

To this end we will partition $S\setminus S_1$ into two parts $X\cup Y$, and estimate $e_r(S)$ from above by breaking it up as follows:
\begin{align}
e_r(S) = e_r(S_1) &+ e_r(X) + e_r(S_1,X) + e_r(X,S_1)\notag \\
+& e_r(Y) + e_r(S_1\cup X,Y)+e_r(Y,S_1\cup X)\text{.} \label{sestimate}
\end{align}
Write $n_1:=\card{S_1}$.
To define $Y$ we introduce the following notion.
A circle $C$ on $\Sigma_1$ is said to be \define{rich} if $\card{S_1\cap C}\geq n_1/8$.
If there are at least $5$ rich circles on $\Sigma_1$, inclusion-exclusion gives (since two circles intersect in at most two points) that
\[\left(\frac{1}{2}+\frac{2\epsi}{3}\right)n_1>\card{S_1\cap\Sigma_1}\geq 5\frac{n_1}{8}-2\binom{5}{2}\text{,}\]
which leads to a contradiction for sufficiently large $n_1>n/4$.

Therefore, there are at most $4$ rich circles on $\Sigma_1$.
Let $Y$ be the set of all points in $(S\setminus S_1)\cap V_1$ that are equidistant to some rich circle.
Let $X:=S\setminus(S_1\cup Y)$.
Write $x:=\card{X}$ and $y:=\card{Y}$.
Note that $Y$ can be covered by $4$ lines, since the points in the $3$-dimensional $V_1$ that are equidistant to some circle all lie on a line.
Since any point is equidistant to at most $2$ points on a line, we have $e_r(\vx,Y)\leq 8$ for all $\vx\in S$, hence
\begin{align}\label{bestimate}
e_r(Y)+e_r(S_1\cup X, Y)+e_r(Y,S_1\cup X)&\leq 8y+8(n_1+x)+y(n_1+x)\notag\\
&=8n+yn_1+yx\text{.}
\end{align}
To bound $e_r(S_1,X)+e_r(X,S_1)$ from above, we estimate $e_r(S_1,\vx)+e_r(\vx,S_1)$ for $\vx\in X$.
If $r(\vx)=1$ then $e_r(S_1,\vx)+e_r(\vx,S_1)=0$ since $S_1$ is the vertex set of a connected component of the graph $G_r^2(S)$ of double edges and $\vx\notin S_1$.
Thus we may assume without loss of generality that $r(\vx)\neq 1$.

If $e_r(S_1\cap C_2,\vx)+e_r(\vx,S_1\cap C_2)\leq 4$, then
\begin{align*}
&\phantom{{}={}} e_r(S_1,\vx)+e_r(\vx,S_1)\\
&= e_r(S_1\cap C_2,\vx)+e_r(S_1\setminus C_2,\vx)+e_r(\vx,S_1\cap C_2)+e_r(\vx,S_1\setminus C_2)\\
&\leq 4+e_r(S_1\setminus C_2,\vx)+e_r(\vx,S_1\setminus C_2)\\
&\leq 4+\card{S_1\setminus C_2}=4+\card{S_1\cap\Sigma_2}+\card{T}\\
&< 4+\left(\frac{1}{2}+\frac{2\epsi}{3}\right)n_1 + \frac{\epsi}{3}n_1 = 4+\left(\frac{1}{2}+\epsi\right)n_1\\
&< \left(\frac{3}{4}+\epsi\right)n_1\quad\text{for $n_1>n/4$ sufficiently large.}
\end{align*}
Otherwise $e_r(S_1\cap C_2,\vx)+e_r(\vx,S_1\cap C_2)\geq 5$, and then either $e_r(S_1\cap C_2,\vx)\geq 3$ or $e_r(\vx,S_1\cap C_2)\geq 3$.
In both cases $\vx$ is equidistant to $C_2$, which implies that $\vx\in V_1$.
Also,
\[e_r(S_1\cap C_2,\vx)+e_r(\vx,S_1\cap C_2)=\card{S_1\cap C_2}<\left(\frac{1}{2}+\frac{2\epsi}{3}\right)n_1\text{.}\]
If $\vx\neq\vo$, then the points on $\Sigma_1$ at distance $1$ to $\vx$ lie on a circle.
Since $\vx\notin Y$, this circle is not rich, giving $e_r(S_1\cap\Sigma_1,\vx)<n_1/8$.
Similarly, $e_r(\vx,S_1\cap\Sigma_1)<n_1/8$.
Putting everything together, we obtain
\begin{align*}
&\phantom{{}={}} e_r(S_1,\vx)+e_r(\vx,S_1)\\
&= e_r(S_1\cap C_2,\vx)+e_r(\vx,S_1\cap C_2)+e_r(S_1\cap\Sigma_1,\vx)+e_r(\vx,S_1\cap\Sigma_1)\\
&\qquad+e_r(T,\vx)+e_r(\vx,T)\\
&< \left(\frac{1}{2}+\frac{2\epsi}{3}\right)n_1 + \frac{n_1}{8} + \frac{n_1}{8} + \frac{\epsi}{3}n_1\\
&=\left(\frac{3}{4}+\epsi\right)n_1\text{.}
\end{align*}
We have shown that for all $\vx\in X\setminus\set{\vo}$,
\[e_r(S_1,\vx)+e_r(\vx,S_1)<\left(\frac{3}{4}+\epsi\right)n_1\text{.}\]
Also, if $\vo\in X$ then $e_r(S_1,\vo)+e_r(\vo,S_1)\leq\card{S_1}=n_1$.
Sum over over all $\vx\in X$ to obtain
\begin{equation}\label{cestimate}
e_r(S_1,X)+e_r(X,S_1)<\left(\frac{3}{4}+\epsi\right)n_1x+\frac{1}{4}n_1\text{.}
\end{equation}
Since $n_1>n/4$ is sufficiently large,
\begin{equation}\label{destimate}
e_r(S_1)<\left(\frac{1}{2}+\delta\right)n_1^2\text{.}
\end{equation}
Recall that we want to show that $\card{S\setminus S_1}=\card{X\cup Y}<2\epsi n/3$.
Suppose that $x\geq \epsi n/4$.
Since $n$ is sufficiently large,
\begin{equation}\label{eestimate}
e_r(X)<\left(\frac{1}{2}+\delta\right)x^2\text{.}
\end{equation}
Substitute \eqref{bestimate}, \eqref{cestimate}, \eqref{destimate} and \eqref{eestimate} into \eqref{sestimate} to obtain
\begin{align*}
\left(\frac{1}{2}-\delta\right)n^2 < e_r(S)&< \left(\frac{1}{2}+\delta\right)n_1^2 + \left(\frac{1}{2}+\delta\right)x^2 + \left(\frac{3}{4}+\epsi\right)n_1x+\frac{1}{4}n_1\\
&\phantom{{}={}} +8n +yn_1 + yx\\ 
&= \left(\frac{1}{2}+\delta\right)(n_1+x+y)^2 - \left(\frac{1}{2}+\delta\right)y^2 - 2\delta yn_1  - 2\delta yx\\
&\phantom{{}={}} - \left(\frac{1}{4}+2\delta-\epsi\right)n_1x + 8n + \frac{1}{4}n_1\\
&< \left(\frac{1}{2}+\delta\right)n^2 - \left(\frac{1}{2}+\delta\right)y^2 - \left(\frac{1}{4}-\frac{1}{20}\right)\frac{n}{4}\cdot\frac{\epsi n}{4}+ 9n\text{.}\\
\intertext{It follows that}
\frac{\epsi}{80}n^2 &< 2\delta n^2 + 9n < 3\delta n^2
\end{align*}
for $n$ sufficiently large, hence $\delta > \epsi/240$, a contradiction.
Therefore, $x < \epsi n/4$.
Now substitute \eqref{bestimate}, \eqref{cestimate}, \eqref{destimate} and the trivial $e_r(X)<x^2$ into \eqref{sestimate} to obtain
\begin{align*}
\left(\frac{1}{2}-\delta\right)n^2 < e_r(S)&< \left(\frac{1}{2}+\delta\right)n_1^2 + x^2 + \left(\frac{3}{4}+\epsi\right)n_1x+\frac{1}{4}n_1\\
&\phantom{{}={}} +8n +yn_1 + yx\\ 
&= \left(\frac{1}{2}+\delta\right)(n_1+x+y)^2 + \left(\frac{1}{2}-\delta\right)x^2 - \left(\frac{1}{2}+\delta\right)y^2 \\
&\phantom{{}={}} - 2\delta yn_1  - 2\delta yx- \left(\frac{1}{4}+2\delta-\epsi\right)n_1x + 8n + \frac{1}{4}n_1\\
&< \left(\frac{1}{2}+\delta\right)n^2 +\left(\frac{1}{2}-\delta\right)x^2 - \left(\frac{1}{2}+\delta\right)y^2 + 9n\text{.}\\
\intertext{from which it follows that}
\left(\frac{1}{2}+\delta\right)y^2 &< 2\delta n^2 + \left(\frac{1}{2}-\delta\right)x^2 +\delta n^2 < 3\delta n^2 + \frac{1}{2}\left(\frac{\epsi n}{4}\right)^2\\
\intertext{for $n$ sufficiently large, and}
y^2 &< \left(\left(\frac{\epsi}{4}\right)^2+6\delta\right)n^2 < \left(\frac{\epsi}{3}\right)^2n^2\text{.}
\end{align*}
Thus $y<\epsi n/3$, and it follows that $\card{X\cup Y}=x+y<\epsi n/4 +\epsi n/3$, which finishes the proof of Theorem~\ref{theorem:4dstability}.
\qed

\section{Proof of Theorem~\ref{theorem:higherd}}\label{proof2}
By Theorem~\ref{theorem:4dstability}, extremal favourite distance digraphs are unit distance graphs after scaling and up to an exceptional set $S_0$ of $o(n)$ points.
Similarly, by removing a set $S_0$ of $o(n)$ points from an extremal furthest neighbour digraph, we obtain a maximum distance graph.
(Note that none of the furthest distances change when restricted to the Lenz configuration $S\setminus S_0$.)

Our goal is to show that there are in fact no exceptional points in an extremal configuration, that is, that $\left.r\right|_{S_0}\equiv 1$.
We do this by some careful counting.
In particular, we need to understand how quickly the functions $u_d(n)$ and $M_d(n)$ grow, that is, we need lower bounds for $u_d(n)-u_d(n-k)$ and $M_d(n)-M_d(n-k)$ where $k$ is small.
The exact values of $u_d(n)$ and $M_d(n)$ are known for all sufficiently large $n$ depending on $d$, except in the case of $u_d(n)$ for odd $d\geq 5$.
Thus in these cases we may simply calculate.
Where we don't know the exact values, we have to use our knowledge of the structure of extremal unit distance and diameter graphs (Theorem~\ref{theorem:udextremal}).

In the next two lemmas we state the values for $M_d(n)$ as well as $u_4(n)$.
(The exact values of $u_d(n)$ for even $d\geq 6$ can be found in \cite{Sw09}.)
Here $t_p(n)$ denotes the number of edges of a \emph{Tur\'an $p$-partite graph} on $n$ vertices, that is, of a complete $p$-partite graph with the $n$ vertices divided into $p$ parts as equally as possible.
\begin{lemma}[Brass \cite{Brass97}, Van Wamelen \cite{VanWamelen}]\label{exactvalues1}
For all $n\geq 5$,
\begin{align*}
u_4(n) &= \begin{cases}
t_2(n) + n & \text{if $n$ is divisible by $8$ or $10$,}\\
t_2(n) + n - 1 & \text{otherwise.}
\end{cases}
\end{align*}
\end{lemma}
\begin{lemma}[\cite{Sw09}]\label{exactvalues2}
For all sufficiently large $n$ \textup{(}depending on $d$\textup{)},
\begin{align*}
M_4(n) &= \begin{cases}
t_2(n)+\lceil n/2\rceil + 1 & \text{if}\quad n\not\equiv 3\pmod{4},\\
t_2(n)+\lceil n/2\rceil    & \text{if}\quad n\equiv 3\pmod{4};
\end{cases}\\
M_5(n) &= t_2(n) +n;\\
M_d(n) &= t_p(n) + p\quad\text{for even $d\geq 6$, where $p=d/2$;}\\
M_d(n) &= t_p(n) + \lceil n/p\rceil +p-1\quad\text{for odd $d\geq 7$, where $p=\lfloor d/2\rfloor$.}
\end{align*}
\end{lemma}
\begin{lemma}\label{nklemma}
For any $d\geq 4$ there exists $N=N_d\geq 1$ such that for any $n$ and $k$ such that $n>k\geq 1$ and $n-k\geq N$,
\begin{equation}\label{udestimate}
u_d(n)-u_d(n-k)\geq \left(1-\frac{1}{p}\right)k(n-k)\text{,}
\end{equation}
\begin{equation}\label{Mdestimate}
M_d(n)-M_d(n-k)\geq \left(1-\frac{1}{p}\right)k(n-k)\text{,}
\end{equation}
\begin{equation}\label{u5estimate}
u_5(n)-u_5(n-k)\geq \frac{1}{2}k(n-k)+\frac{k^2+2k-1}{4}\text{,}
\end{equation}
and for $n-k\geq N_5$,
\begin{equation}\label{M5estimate}
M_5(n)-M_5(n-k)\geq \frac{1}{2}k(n-k)+\frac{k^2+4k-1}{4}\text{.}
\end{equation}
Also, for $n$ sufficiently large,
\begin{equation}\label{u41estimate}
u_4(n)-u_4(n-1) = \frac{n-1}{2}\quad\text{only if $8\divides n-1$ or $10\divides n-1$.}
\end{equation}
\end{lemma}
\begin{proof}
Since Lemmas~\ref{exactvalues1} and \ref{exactvalues2} provide the exact values of $u_4(n)$ and $M_d(n)$, $d\geq 4$, the inequalities~\eqref{Mdestimate}, \eqref{M5estimate} and \eqref{u41estimate} can be obtained by simple calculations.
We omit the details, except to note that $t_p(n)-t_p(n-k)\geq (1-1/p)k(n-k)$, as can be seen by taking a Tur\'an $p$-partite graph on $n-k$ vertices and adding $k$ new vertices to the smallest class.

We next prove the remaining inequalities \eqref{udestimate} and \eqref{u5estimate}.
Since these all involve $u_d(n)$ for which we do not have exact values when $d\geq 5$ is odd, we give a structural argument.
Consider a set $S$ of $n-k$ points in $\bR^d$ that is extremal with respect to unit distances, that is, $u(S)=u_d(n-k)$.
By Theorem~\ref{theorem:udextremal}, $S$ is a Lenz configuration if $n-k$ is sufficiently large.
In particular, $S$ can be partitioned into $p=\lfloor d/2\rfloor$ parts $S_1,\dots,S_p$ with each part lying on a circle (except if $d$ is odd when $S_1$ lies on a sphere) such that the distance between any two points on different circles (on a circle and the sphere, respectively), equals $1$.
Let $i\in[p]$ be such that $\card{S_i}=\min\set{\card{S_1},\dots,\card{S_p}}$.
Thus $\card{S_i}\leq (n-k)/p$.
Choose a set $T$ of any $k$ points on the circle (or sphere) containing $S_i$ disjoint from $S_i$.
Then $S\cup T$ contains $n$ points and has at least $k\card{S\setminus S_i}\geq k(1-1/p)(n-k)$ additional unit distance pairs.
This establishes \eqref{udestimate}.

Next consider \eqref{u5estimate}.
Here $S\subset\Sigma_1\cup C_2$, where $\Sigma_1$ is a $2$-sphere and $C_2$ a circle, with any point on $\Sigma_1$ and any point on $C_2$ at unit distance.
Now add $k$ new points to $S$ in the following more careful way.
If $k$ is even, add $k/2$ points to each of $\Sigma_1$ and $C_2$.
This creates
\[\frac{k}{2}\card{S\cap\Sigma_1}+\frac{k}{2}\card{S\cap C_2}=k(n-k)/2\]
unit distance pairs from the new points to $S$ and $k^2/4$ unit distance pairs between the new points.
Since $r_1>1/2$ (otherwise $u_d(n-k)=u(S)\leq t_2(n-k)+O(n)$, a contradiction), it is possible to choose each new point on $\Sigma_1$ at unit distance to some point of $S\cap\Sigma_1$.
We obtain a set $S'$ of $n$ points with at least
\[ u_5(n-k)+\frac{1}{2}k(n-k)+\frac{k^2}{4}+\frac{k}{2}\]
unit distance pairs.
Therefore, 
\[ u_5(n)\geq u(S')\geq u_5(n-k)+\frac{1}{2}k(n-k)+\frac{k^2}{4}+\frac{k}{2}\text{,}\]
which proves \eqref{u5estimate} when $k$ is even.
Now let $k$ be odd.
If we place $(k-1)/2$ points on $\Sigma_1$ and $(k+1)/2$ points on $C_2$, this creates as before
\begin{align}
&\phantom{{}={}}\frac{k-1}{2}\card{S\cap C_2}+\frac{k+1}{2}\card{S\cap\Sigma_1}+\frac{k^2-1}{4}+\frac{k-1}{2}\notag\\
&= \frac{1}{2}k(n-k)+\frac{1}{2}(\card{S\cap\Sigma_1}-\card{S\cap C_2})+\frac{k^2-1}{4}+\frac{k-1}{2}\label{odd1}
\end{align}
additional unit distance pairs.
If instead we place $(k+1)/2$ points on $\Sigma_1$ and $(k-1)/2$ points on $C_2$,
the number of additional unit distance pairs created is
\begin{align}
&\phantom{{}={}}\frac{k+1}{2}\card{S\cap C_2}+\frac{k-1}{2}\card{S\cap\Sigma_1}+\frac{k^2-1}{4}+\frac{k+1}{2}\notag\\
&= \frac{1}{2}k(n-k)+\frac{1}{2}(\card{S\cap C_2}-\card{S\cap \Sigma_1})+\frac{k^2-1}{4}+\frac{k+1}{2}\text{.}\label{odd2}
\end{align}

It is always possible to attain at least the average of \eqref{odd1} and \eqref{odd2}, which equals the right-hand side of \eqref{u5estimate}.
\end{proof}

In the above lemma it is tempting to try to prove the inequalities \eqref{Mdestimate} and \eqref{M5estimate} also with the use of Theorem~\ref{theorem:udextremal}.
However, we should then be careful in how we choose the $k$ points to be added to the extremal Lenz configuration on $n-k$ points, so as not to change furthest distances among the original $n-k$ points.
Although this is possible, the case $d=5$ requires a very detailed consideration of the extremal $5$-dimensional diameter graphs as determined in \cite{Sw09}. 
It is much simpler to instead use the estimates from Lemma~\ref{exactvalues2} and calculate.

\bigskip
We can now start with the proof of Theorem~\ref{theorem:higherd}.
Let $d\geq 4$, $p=\lfloor d/2\rfloor$ and let $S\subset\bR^d$ with $\card{S}=n$ and $r\colon S\to(0,\infty)$ determine an extremal favourite distance digraph (or let $S$ determine an extremal furthest neighbour digraph respectively, and then continue to write $r(\vx)=D_S(\vx)$ for $\vx\in S$).

Apply Theorem~\ref{theorem:4dstability}.
Thus if $n$ is sufficiently large depending on $d$, $\bR^d$ has an orthogonal decomposition $V_1\oplus\dots V_p$ with $\dim V_i=2$ (except when $d$ is odd, $\dim V_1=3$) such that after scaling and translation of $S$, there is a Lenz system $(C_1,\dots,C_p)$ for $d$ even, $(\Sigma_1,C_2,\dots,C_p)$ for $d$ odd, and a partition $S_0,S_1,\dots,S_p$ of $S$ with $\card{S_0}=o(n)$, $\card{S_i}=\frac{n}{p}+o(n)$ and $S_i\subset C_i$ where $C_i$ is a circle with centre $o$ and radius $r_i$ in $V_i$ for $i\in[p]$ (except if $d$ is odd and $i=1$, where $S_1\subset\Sigma_1\subset V_1$) such that $r_i^2+r_j^2=1$ for all distinct $i,j$.
Also, $\left. r\right|_{S\setminus S_0}\equiv 1$.

Let $T:=\setbuilder{\vx\in S_0}{r(\vx)\neq 1}$.
If we can show that $T=\emptyset$, then $r\equiv 1$ and $S$ would consequently determine an extremal unit distance graph (extremal diameter graph, respectively), since $2u(S)=e_r(S)\geq 2u_d(n)$ (respectively $2M(S)=e_r(S)\geq 2M_d(n)$).
It would then follow from Theorem~\ref{theorem:udextremal} that $S$ is a Lenz configuration for sufficiently large $n$.
In the exceptional case of favourite distances in dimension $4$, we show instead that if $T\neq\emptyset$ then $T=\set{\vo}$.
As in the proof of Theorem~\ref{theorem:4dstability}, the dimensions $d\geq 6$ are disposed of very quickly, and the case $d=5$ takes the most work.

Write $k:=\card{T}$.
Suppose that $k\neq 0$.
We aim to find a contradiction except in the $4$-dimensional case, where we'll prove that $k=1$ and $T=\{\vo\}$, $r(\vo)=r_1=r_2=1/\sqrt{2}$.

We estimate as follows:
\begin{align}
2u_d(n)\leq e_r(S)&=e_r(S\setminus T)+e_r(S\setminus T,T)+e_r(T,S\setminus T)+e_r(T)\notag\\
&\leq 2u_d(n-k)+e_r(S\setminus T,T)+e_r(T,S\setminus T)+e_r(T)\text{.}\label{bound0}
\end{align}
This, together with \eqref{udestimate} of Lemma~\ref{nklemma} gives
\begin{align*}
2\left(1-\frac{1}{p}\right)k(n-k) &\leq 2u_d(n)-2u_d(n-k)\notag\\
&\leq e_r(S\setminus T,T)+e_r(T,S\setminus T)+e_r(T)\text{.}
\end{align*}
Using instead \eqref{Mdestimate} (for the case of furthest neighbours) gives the same bounds, so in all cases we have
\begin{equation}\label{bound1}
2\left(1-\frac{1}{p}\right)k(n-k)\leq e_r(S\setminus T,T)+e_r(T,S\setminus T)+e_r(T)\text{.}
\end{equation}
Since $r(\vx)\neq 1$ for all $\vx\in T$, $\vx$ is not adjacent to any point from $S\setminus T$ in the graph $G_r^2(S)$ of double edges, so
\begin{equation}\label{bound2}
e_r(S\setminus T,T)+e_r(T,S\setminus T)\leq k(n-k)\text{.}
\end{equation}
Substituting this and the trivial bound $e_r(T)\leq k(k-1)$ into \eqref{bound1} we obtain
\[ 2\left(1-\frac{1}{p}\right)k(n-k) \leq k(n-1)\text{.} \]
Since $k=o(n)$, we obtain a contradiction for sufficiently large $n$ if $p\geq 3$.
This finishes the proof for the cases $d\geq 6$.

\smallskip
Now assume that $d\in\set{4,5}$.
Suppose that for some $x\in T$, $e_r(S_2,\vx)+e_r(\vx,S_2)\leq 4$.
Then we may improve \eqref{bound2} to
\begin{equation*}
e_r(S\setminus T,T)+e_r(T,S\setminus T)\leq k(n-k)-\card{S_2}+4\text{.}
\end{equation*}
Substituting this and $e_r(T)\leq k(k-1)$ into \eqref{bound1} we obtain
$k(n-k)\leq k(n-1)-\card{S_2}+4$, hence $\card{S_2}\leq 4+k(k-1)$.
This contradicts $\card{S_2}=n/2+o(n)$ for $n$ sufficiently large.

Therefore, for all $\vx\in T$ we have $e_r(S_2,\vx)+e_r(\vx,S_2)\geq 5$, which implies either $e_r(S_2,\vx)\geq 3$ or $e_r(\vx,S_2)\geq 3$.
Either case gives that $\vx$ is equidistant to the circle $C_2$.
Therefore, $x\in V_1$.

We have shown that $T\subset V_1$.

\smallskip
We can now finish the case $d=4$.
Symmetry gives that $T\subset V_2$ as well, hence $T=\set{\vo}$.
Therefore, $\vo$ must have the same distance $r(\vo)$ to $C_1$ and $C_2$, and it follows that $r(\vo)=r_1=r_2=1/\sqrt{2}$ and $e_r(S\setminus\set{\vo},\vo)+e_r(\vo,S\setminus\set{\vo})=n-1$.
In the favourite distance case we obtain from \eqref{bound0} that
$u_4(n)-u_4(n-1)\leq(n-1)/2$.
Combined with \eqref{udestimate} of Lemma~\ref{nklemma}, we obtain that equality holds, hence $8\divides n-1$ or $10\divides n-1$ by \eqref{u41estimate}, and $S\setminus\set{\vo}$ is an extremal unit distance configuration.
Inspection of the extremal configurations \cite{Brass97, VanWamelen} shows that when $8\notdivides n-1$ and $10\divides n-1$, the two circles are necessarily of different radii.
In our case we must therefore have $8\divides n-1$.
Then the extremal unit distance configurations on $n-1$ points are formed by the vertices of $(n-1)/8$ unit squares inscribed in each $C_i$ \cite{Brass97}.

In the furthest neigbour case, we obtain similarly as above that $M_4(n)-M_4(n-1)\leq(n-1)/2$.
Again, by \eqref{Mdestimate} equality holds and $S\setminus\set{\vo}$ is an extremal diameter configuration.
However, it is easily seen that when $r_1=r_2=1/\sqrt{2}$, the maximum number of diameter pairs in a set of $n$ points in $C_1\cup C_2$ is $t_2(n)+2$, which contradicts Lemma~\ref{exactvalues2} for sufficiently large $n$.
This finishes the proof for the case $d=4$.

\smallskip
Now consider the case $d=5$.
Suppose that $e_r(S_1,\vx)+e_r(\vx,S_1)<n/3$ for some $\vx\in T$.
Then we may improve \eqref{bound2} to
\begin{equation*}
e_r(S\setminus T,T)+e_r(T,S\setminus T) < k(n-k)-\card{S_1}+\frac{n}{3}\text{,}
\end{equation*}
which, when substituted together with $e_r(T)\leq k(k-1)$ into \eqref{bound1}, gives
$k(n-k) < k(n-1)-\card{S_1}+n/3$ and subsequently, $\card{S_1} < n/3 +k(k-1)$, which contradicts $\card{S_1}=n/2+o(n)$ for $n$ sufficiently large.

Therefore, for all $\vx\in T$ we have $e_r(S_1,\vx)+e_r(\vx,S_1)\geq n/3$.
This will enable us to show that $T$ lies on a straight line through the origin.
Suppose then that for some two $\vx,\vx'\in T\setminus\set{\vo}$, the lines $\vo\vx$ and $\vo\vx'$ are not parallel.
Then at least $n/3$ points of $S_1$ lie on two circles that are both normal to $\vo\vx$, and similarly, at least $n/3$ points of $S_1$ lie on two circles normal to $\vo\vx'$.
Since the intersection of these two unions of circles contains at most $8$ points, we obtain
\[n/2+o(n)=\card{S_1}\geq 2\cdot\frac{n}{3}-8\text{,}\]
a contradiction for sufficiently large $n$.

It follows that $T$ lies on a line $\ell$, say, through the origin.
Since there are at most $2$ points on $\ell$ at distance $r(\vx)$ to $\vx$, it follows that $e_r(\vx,T)\leq 2$ for all $\vx\in T$, and when $\vx$ is the first or last point of $T$ on $\ell$, $e_r(\vx,T)\leq 1$.
It follows that $e_r(T)\leq 2k-2$ (keeping in mind that $T\neq\emptyset$ by assumption).

In the case of extremal favourite distance digraphs, bounds~\eqref{bound0}, \eqref{bound2} and $e_r(T)\leq 2k-2$, together with \eqref{u5estimate} of Lemma~\ref{nklemma} give
\[ k(n-k) + \frac{k^2+2k-1}{2} \leq 2u_5(n)-2u_5(n-k)\leq k(n-k)+2k-2\text{,}\]
which simplifies to $(k-1)^2+2\leq 0$, a contradiction.

For extremal furthest neighbour digraphs, a similar calculation (now using \eqref{M5estimate} instead of \eqref{u5estimate}) gives that
\[ k(n-k) + \frac{k^2+4k-1}{2} \leq 2M_5(n)-2M_5(n-k)\leq k(n-k)+2k-2\text{,}\]
which simplifies to $k^2+3\leq 0$, another contradiction.

We have shown that $k=0$ in all cases when $d=5$, and it follows that $S$ is a Lenz construction.
\qed

\section*{Acknowledgement}
I thank the anonymous referee for careful proofreading and good advice on a previous version.

\end{document}